\documentclass[11pt,reqno,sumlimits]{amsart}
\usepackage{amsfonts, amsmath, amscd, amssymb, euscript, amsthm, array, booktabs,
dcolumn, shortvrb, tabularx, units, url, mathrsfs}
\usepackage[pdftex]{graphicx}
\usepackage[all]{xy}
\normalsize
\DeclareMathOperator{\Pf}{Pf}

\DeclareMathOperator{\id}{id}

\DeclareMathOperator{\SO}{SO}

\begin{document}
\setlength{\baselineskip}{1.4\baselineskip}
\theoremstyle{definition}
\newtheorem{defi}{Definition}
\newtheorem{remark}{Remark}
\newtheorem{coro}{Corollary}
\newtheorem{exam}{Example}
\newtheorem{thm}{Theorem}
\newtheorem{prop}{Proposition}
\newtheorem{lemma}{Lemma}
\numberwithin{equation}{section}
\newcommand{\wt}[1]{{\widetilde{#1}}}
\newcommand{\ov}[1]{{\overline{#1}}}
\newcommand{\wh}[1]{{\widehat{#1}}}
\newcommand{\poin}{Poincar\'e~}
\newcommand{\deff}[1]{{\bf\emph{#1}}}
\newcommand{\boo}[1]{\boldsymbol{#1}}
\newcommand{\abs}[1]{\lvert#1\rvert}
\newcommand{\norm}[1]{\lVert#1\rVert}
\newcommand{\inner}[1]{\langle#1\rangle}
\newcommand{\poisson}[1]{\{#1\}}
\newcommand{\biginner}[1]{\Big\langle#1\Big\rangle}
\newcommand{\set}[1]{\{#1\}}
\newcommand{\Bigset}[1]{\Big\{#1\Big\}}
\newcommand{\BBigset}[1]{\bigg\{#1\bigg\}}
\newcommand{\dis}[1]{$\displaystyle#1$}
\newcommand{\V}{\mathcal{V}}
\newcommand{\R}{\mathbb{R}}
\newcommand{\N}{\mathbb{N}}
\newcommand{\Z}{\mathbb{Z}}
\newcommand{\Q}{\mathbb{Q}}
\newcommand{\h}{\mathbb{H}}
\newcommand{\g}{\mathfrak{g}}
\newcommand{\C}{\mathbb{C}}
\newcommand{\RRR}{\mathscr{R}}
\newcommand{\DDD}{\mathscr{D}}
\newcommand{\so}{\mathfrak{so}}
\newcommand{\gl}{\mathfrak{gl}}
\newcommand{\LL}{\mathcal{L}}
\newcommand{\BB}{\mathcal{B}}
\newcommand{\CC}{\mathcal{C}}
\newcommand{\HH}{\mathcal{H}}
\newcommand{\G}{\mathcal{G}}
\newcommand{\sss}{\mathbb{S}}
\newcommand{\E}{\mathcal{E}}
\newcommand{\EE}{\mathscr{E}}
\newcommand{\UU}{\mathcal{U}}
\newcommand{\F}{\mathcal{F}}
\newcommand{\cdd}[1]{\[\begin{CD}#1\end{CD}\]}
\normalsize
\title{Gauss--Bonnet--Chern theorem and differential characters}
\author{Man-Ho Ho}
\address{Department of Mathematics\\ Hong Kong Baptist University}
\email{homanho@hkbu.edu.hk}
\maketitle
\nocite{*}
\begin{abstract}
In this paper we first prove that every differential character can be represented by
differential form with singularities. Then we lift the Gauss--Bonnet--Chern theorem
for vector bundles to differential characters.
\end{abstract}
\tableofcontents
\section{Introduction}

The purpose of this paper is to prove that every differential character
\cite{CS85, BB14} can be represented by differential form with singularities and to
prove a version of the Gauss--Bonnet--Chern theorem (GBC theorem) for vector bundles
taking values in differential characters.

The subject of study in this paper is differential characters, whose philosophy can
be traced back to Chern's intrinsic proof of the GBC theorem \cite{C44, C45}. The
main idea is transgression form. For a given even dimensional Riemannian manifold $X$
Chern constructs a differential form on the total space of the sphere bundle of
$TX\to X$ so that its differential is the pullback of the Euler form of $X$. Since
then transgression form had found lots of applications. Among them, transgression
form is used to give integral formula of characteristic classes (see
\cite{C44a, A50, T53, E59, W76a, W76b, M80, M86} for example). An instructive example
is the Stiefel--Whitney class. For a given $k\in\N$ and an integral domain $A$ of
$\R$, Allendoerfer and Eells define the notion of an $(A, k)$-pair of differential
forms with singularities \cite{AE58} on a manifold $X$ (the precise definition is
given in \S2.2). They use it to define de Rham cohomology of $X$ with coefficients in
$A$ and prove an analogue of the de Rham theorem. Then Eells gives an integral formula
of the $k$-th Stiefel--Whitney class of orientable Riemannian manifolds
\cite[Theorem 4D]{E59} using a particular $(\Z, k)$-pair of forms, which are certain
pullbacks of the $k$-th universal Gauss curvature form and the $(k-1)$-th universal
geodesic curvature form.

Using Chern--Weil theory, Chern and Simons are able to express any given characteristic
form, which lives in the total space of the principal bundle, in terms of the exterior
differential of its transgression form \cite[Proposition 3.2]{CS74}. It turns out that
such a transgression form defines a cohomology class of the principal bundle in certain
cases (see \cite[Theorem 3.9]{CS74} for example). Thus it can be regarded as a secondary
invariant of principal bundle. One of the motivations of defining differential
characters is to obtain secondary invariants of principal bundle living in the base
\cite[Proposition 2.8]{CS85}. Moreover, the Euler class, the Chern classes and the
Pontrjagin classes of vector bundles can be lifted to differential characters and each
of them can be represented by a pair of characteristic form and its transgression form
(in a certain sense) via some integral formulas \cite[\S3-5]{CS85}. Then a natural
question is whether every differential character can be represented by a pair of
differential forms (possibly with singularities, for example the integral formula of
Stiefel--Whitney class given by Eells mentioned above) via integral formula.

Cheeger answers the question by stating (without proof) that every differential
character can be represented by differential form with singularities
\cite[Proposition 3]{C73}. Motivated by deriving formulas for the Abel--Jacobi map of
groups of cycles on a complex algebraic variety to intermediate Jacobians, Harris
\cite{H89} gives a proof of \cite[Proposition 3]{C73} under additional assumptions.
In the other direction, B\"ar--Becker \cite{BB14} show that every differential
character can be written as a sum of two integrals of differential forms when the
underlying manifold is extended to a large class of stratified spaces, called
stratifolds. In this paper we prove \cite[Proposition 3]{C73} in full generality
(Proposition \ref{prop 1}). More precisely, we prove that any differential character
$f\in\wh{H}^k(X; \R/A)$ can be represented by an $(A, k)$-pair of differential forms
with singularities.

The answer to the above question leads us to ask whether equalities of characteristic
classes (or forms) can be lifted to the level of differential characters. We show that
the GBC theorem does have such a lift (Theorem \ref{thm 1}). This is not surprising
because Chern's intrinsic proof of the GBC theorem is the starting point of secondary
invariants. Thus the GBC theorem ought to take place in differential characters.
However, we cannot find such a statement nor a proof in literature so we feel it is
worthwhile to write it down. Another example of such lifts is the differential
Grothendieck--Riemann--Roch  theorem \cite{B05, BS09, FL10}.

The version of the GBC theorem considered here states that for a given oriented
Euclidean vector bundle $(E, h^E, \nabla^E)$ with a Euclidean metric $h^E$ and a
metric-compatible connection $\nabla^E$ we have the following equality in differential
characters
$$\wh{\chi}(\nabla^E)-v^*\wh{U}(h^E, \nabla^E)=i_2(\omega),$$
where $\wh{\chi}$ is the differential Euler class, $\wh{U}^*$ is the differential
Thom class, and $v$ is an arbitrary but fixed section of $E\to X$ (the remaining
details are given in \S2.3). Another purpose of stating and proving the GBC theorem
taking values in differential characters is to serve as a first step to extend the
GBC theorem for (at least) flat vector bundles over manifolds with boundary
\cite[Theorem 3.2]{BM06} to differential characters.

\subsection{Relation to previous work}

Harvey--Lawson--Zweck \cite{HLZ03} define differential characters in terms of
de Rham--Federer theory of currents, and they show that it is naturally isomorphic to
the Cheeger--Simons' definition of differential characters \cite[Theorem 4.1]{HLZ03}.
Thus \cite[Theorem 4.1]{HLZ03} and our proof of \cite[Proposition 3]{C73} are the same
in spirit, but stated in a different language.

Harvey--Zweck \cite[\S 8]{HZ01} define the differential Euler class $\wh{\chi}$ in
terms of the Euler sparks \cite[Definition 7.3]{HZ01}, and they give a comparison
between the approaches of defining $\wh{\chi}$ in \cite{CS85} and in \cite[\S8]{HZ01}
is given in \cite[Remark 8.9]{HZ01}. Moreover, they prove a refined version of the
GBC theorem \cite[Corollary 5.3]{HZ01} which is stated in the language of currents.
One can apply these results to prove Theorem \ref{thm 1} in the language of
\cite{HLZ03}.

The paper is organized as follows. In section 2 we review the background material
needed to prove the main results, including Cheeger--Simons differential characters,
differential form with singularities and the Mathai--Quillen's Thom form. In section
3 we prove the main results.

\section*{Acknowledgement}

The author would like to thank Christian Becker and Steve Rosenberg for their comments
and valuable suggestions.

\section{Background material}

Throughout this paper, $X$ is a closed manifold, $A$ is a proper subring of $\R$ and
$\pi:E\to X$ is a real oriented vector bundle of rank $2k$. Let $h^E$ be a Euclidean
metric and $\nabla^E$ a Euclidean connection on $E\to X$ respectively. We write these
data as $(\pi:E\to X, h^E, \nabla^E)$.

\subsection{Cheeger--Simons differential characters}

In this subsection we recall Cheeger--Simons differential characters \cite{CS85, BB14}.

Let $k\geq 1$. A degree $k$ differential character with coefficients in $A$ is a group
homomorphism $f:Z_{k-1}(X)\to\R/A$ for which there exists a fixed $\omega_f\in\Omega^k
(X)$ such that for all $c\in C_k(X)$,
$$f(\partial c)=\int_c\omega_f\mod A.$$
The abelian group of degree $k$ differential characters is denoted by $\wh{H}^k(X;
\R/A)$. Note that $\omega_f$ is a closed $k$-form with periods in $A$ and is uniquely
determined by $f\in\wh{H}^k(X; \R/A)$. Denote by $\Omega_A^k(X)$ the group of closed
$k$-forms on $X$ with periods in $A$. Define a map $\delta_1:\wh{H}^k(X; \R/A)\to
\Omega_A^k(X)$ by $\delta_1(f)=\omega_f$. For every $f\in\wh{H}^k(X; \R/A)$ there
exists $T\in C^k(X; \R)$ such that \dis{(\delta T)(c)=\int_c\omega_f\mod A}. Thus there
exists $u\in C^k(X; A)$ such that $\delta T=\omega_f-u$. The cohomology class $[u]\in
H^k(X; A)$ is independent of the choice of $T$ and it satisfies $r([u])=\omega_f$,
where $r$ is induced by the inclusion of coefficients $A\hookrightarrow\R$. Define a
map $\delta_2:\wh{H}^k(X; \R/A)\to H^k(X; A)$ by $\delta_2(f)=[u]$. Every $(k-1)$-form
$\omega$ can be considered as a differential character by integrating over
$(k-1)$-cycles, and is realized by the map \dis{i_2:\frac{\Omega^{k-1}(X)}
{\Omega^{k-1}_A(X)}\to\wh{H}^k(X; \R/A)} defined by \dis{i_2(\omega)(z)=\int_z\omega
\mod A}, where $z\in Z_{k-1}(X)$, is injective.

\subsection{Differential form with singularities}

In this subsection we briefly recall differential form with singularities \cite{AE58}
and its relation with differential characters \cite{C73, H89}.

Let $\varphi$ be a $(k-1)$-form defined on $X\setminus e(\varphi)$, where $e(\varphi)$ is
a closed nowhere dense set. Suppose $\omega$ is a $k$-form defined on $X\setminus
e(\omega)$, where $e(\omega)$ is a closed subset of $e(\varphi)$, such that $\omega$ is
an extension of $d\varphi$ to $X\setminus e(\omega)$. A chain $c\in C_k(X; A)$ satisfying
$\abs{c}\cap e(\omega)=\emptyset$ and $\abs{\partial c}\cap e(\varphi)=\emptyset$ is
called an admissible chain for the pair $(\omega, \varphi)$. The pair $(\omega, \varphi)$
is called an $(A, k)$-pair \cite[\S 3]{AE58} if
\begin{enumerate}
  \item $e(\omega)$ and $e(\varphi)$ lie on smooth locally finite polyhedra of dimensions
        $\leq n-k-1$ and $\leq n-k$ respectively, and
  \item for any admissible $k$-chain $c\in C_k(X; A)$ for the pair $(\omega, \varphi)$,
        the difference
        $$R[(\omega, \varphi), c]:=\int_c\omega-\int_{\partial c}\varphi$$
        lies in $A$.
\end{enumerate}

Define an equivalence relation on the set of $(A, k)$-pairs as follows. Let $(\omega,
\varphi)\sim(\omega', \varphi')$ if $R[(\omega, \varphi), c]=R[(\omega', \varphi'), c]$
for all chains $c\in C_k(X; A)$ admissible for both pairs. Still denote by $(\omega,
\varphi)$ the equivalence class of the $(A, k)$-pair $(\omega, \varphi)$ and by
$\Omega^k(X; A)$ the set of equivalence classes of $(A, k)$-pairs. Note that $\Omega^k
(X; A)$ is an $A$-module.

For any $(A, k)$-pair $(\omega, \varphi)$ with $e(\omega)=\emptyset$, one can define a
differential character as follows \cite{C73}. For any $z_{k-1}\in Z_{k-1}(X)$, there
exists $z_{k-1}'\in Z_{k-1}(X)$ such that $\abs{z_{k-1}'}\cap\abs{e(\varphi)}=\emptyset$
and $z_{k-1}=\partial c_k+z_{k-1}'$ for some $c_k\in C_k(X)$. Define a differential
character $s(\varphi):Z_{k-1}(X)\to\R/A$ associated to $(\omega, \varphi)$ by
\begin{equation}\label{eq 2.1}
s(\varphi)(z_{k-1})=\int_{c_k}\omega+\int_{z_{k-1}'}\varphi\mod A.
\end{equation}
We call $s(\varphi)$ the differential character induced by the differential form $\varphi$
with singularities.

\subsection{Mathai--Quillen's Thom form}\label{s 2.3}

In this subsection we briefly recall the construction of the Mathai--Quillen's Thom form
\cite{MQ86} using the Berezin integral given in \cite[\S 1.6]{BGV}, but follow the sign
convention given in \cite[\S 3.2]{Z01}.

The Euler form $\chi(\nabla^E)\in\Omega^{2k}_\Z(X)$ of $(\pi:E\to X, h^E, \nabla^E)$ is
defined by
$$\chi(\nabla^E)=\frac{1}{(2\pi)^k}\Pf(R^E).$$
Denote by $T:\Omega(X, \Lambda(E))\to\Omega(E)$ the Berezin integral. Consider the
pullback
\cdd{\pi^*E @>>> E \\ @VVV @VVV \\ E @>>\pi> X}
Equip $\pi^*E\to E$ with the Euclidean metric $\pi^*h^E$ and Euclidean connection
$\pi^*\nabla^E$. The tautological section $\bold{x}\in\Gamma(E, \pi^*E)$ is defined by
$\bold{x}(e)=e\in(\pi^*E)_e\cong E_{\pi(e)}$. Define
\begin{equation}\label{eq 2.2}
\Omega:=\frac{1}{2}\norm{\bold{x}}^2+\pi^*\nabla^E\bold{x}-\pi^*R^E\in\Omega(E,
\Lambda(\pi^*E)),
\end{equation}
where $R^E$ is the curvature of $\nabla^E$. The Mathai--Quillen's Thom form $U(h^E,
\nabla^E)\in\Omega^{2k}_\Z(E)$ of $E\to X$ \cite[(1.37)]{BGV} is defined by
$$U(h^E, \nabla^E):=\frac{1}{(2\pi)^k}T(e^{-\Omega}).$$
Note that by pulling back $U(h^E, \nabla^E)$ to the unit ball bundle of $E$ by the map
\dis{y\mapsto\frac{y}{\sqrt{1-\norm{y}^2}}}, it has compact support.

The transgression formula for the Mathai--Quillen's Thom form when the metric is being
rescaled by $t$ is given as follows. Rescale $\Omega$ in (\ref{eq 2.2}) by
$$\Omega_t=\frac{1}{2}t^2\norm{\bold{x}}^2+t\pi^*\nabla^E\bold{x}-\pi^*R^E.$$
Write \dis{U_t(h^E, \nabla^E)=\frac{1}{(2\pi)^k}T(e^{-\Omega_t})}. By
\cite[Proposition 3.5]{Z01} we have
\begin{equation}\label{eq 2.3}
\frac{d}{dt}U_t(h^E, \nabla^E)=-\frac{(-1)^{k(2k+1)}}{(2\pi)^k}dT(\bold{x}
e^{-\Omega_t}).
\end{equation}
Write \dis{a(k)=\frac{(-1)^{k(2k+1)}}{(2\pi)^k}}. By integrating both sides of
(\ref{eq 2.3}) we have
\begin{equation}\label{eq 2.4}
\chi(\pi^*\nabla^E)-U_t(h^E, \nabla^E)=d\int^t_0a(k)T(\bold{x}e^{-\Omega_t})dt.
\end{equation}
Consider the sphere bundle $\wt{\pi}:SE\to X$ and pull $E\to X$ back by $\wt{\pi}$:,
\cdd{\wt{\pi}^*E @>>> E \\ @VVV @VV\pi V \\ SE @>>\wt{\pi}> X}
For any section $v\in\Gamma(X, E)$, (\ref{eq 2.3}) implies that
$$\frac{d}{dt}v^*U_t(h^E, \nabla^E)=-a(k)dT(v\wedge e^{-\Omega_{t, v}}),$$
where \dis{\Omega_{t, v}:=\frac{t^2}{2}\norm{v}^2+t\pi^*\nabla^Ev-\pi^*R^E}. Thus
\begin{equation}\label{eq 2.6}
\chi(\nabla^E)-v^*U(h^E, \nabla^E)=d\bigg(\int^1_0a(k)T(v\wedge e^{-\Omega_{t, v}})dt
\bigg),
\end{equation}
where we have used the fact that $v^*\circ\pi^*=(\pi\circ v)^*=\id^*$. As noted in
\cite[p.56]{BGV}, (\ref{eq 2.6}) is a more general version of the GBC theorem for
vector bundles.

\section{Main Results}

In this section we prove the main results in this paper.

\subsection{Representing differential character by form with singularities}

First of all we prove that for any $(A, k)$-pair $(\omega, \varphi)$ the homomorphism
$s(\varphi):Z_{k-1}(X)\to\R/A$ is a well defined differential character of degree $k$.
\begin{lemma}\label{lemma 1}
Let $(\omega, \varphi)\in\Omega^k(X; A)$. The map $s(\varphi)$ given by (\ref{eq 2.1})
defines a differential character of degree $k$.
\end{lemma}
\begin{proof}
We first show that $s(\varphi):Z_{k-1}(X)\to\R/A$ is a well defined map. Suppose
$\wt{z}\in Z_{k-1}(X)$ and $\wt{c}\in C_k(X)$ are such that
$$\wt{z}+\partial\wt{c}=z=\partial c+z'$$
and $\abs{\wt{z}}\cap\abs{e(\varphi)}=\emptyset$. Since
\begin{displaymath}
\begin{split}
\bigg(\int_{\wt{c}}\omega+\int_{\wt{z}}\varphi\bigg)-\bigg(\int_c\omega+\int_{z'}
\varphi\bigg)&=\int_{\wt{c}-c}\omega+\int_{\wt{z}-z'}\varphi\\
&=\int_{\wt{c}-c}\omega+\int_{\partial(c-\wt{c})}\varphi\\
&=\int_{\wt{c}-c}\omega-\int_{\partial(\wt{c}-c)}\varphi\\
&=R[(\omega, \varphi), \wt{c}-c]\in A,
\end{split}
\end{displaymath}
where the integral \dis{\int_{\partial(\wt{c}-c)}\varphi} makes sense because
$\abs{\partial c}\subseteq\abs{c}$ and $\abs{c_1+c_2}\subseteq\abs{c_1}\cup\abs{c_2}$,
it follows that $s(\varphi)$ is a well defined map. Obviously $s(\varphi)$ is a group
homomorphism.

If $z\in Z_{k-1}(X)$ is a coboundary, say $z=\partial c$, it follows from
the definition of $s(\varphi)$ that
$$s(\varphi)(\partial c)=\int_c\omega\mod A.$$
Thus $s(\varphi)$ is a differential character of degree $k$.
\end{proof}

We now give a proof of \cite[Proposition 3]{C73}.
\begin{prop}\label{prop 1}
Every differential character $f\in\wh{H}^k(X; \R/A)$ can be written as a differential
character induced by a differential form with singularities.
\end{prop}
In \cite{H89} Proposition \ref{prop 1} is proved for $A=\Z$ under the assumption that
$\delta_1(f)$ is the \poin dual of a compact oriented $(n-k)$-submanifold of $X$.
\begin{proof}
Write $\omega=\delta_1(f)$ and $[u]=\delta_2(f)$. By \cite[Theorem 5C]{AE58} there exists
a smooth differential form $\varphi$ of degree $k-1$ with singularities lying on an
$(n-k)$-cycle such that $(\omega, \varphi)$ is an $(A, k)$-pair. Thus the induced
differential character $s(\varphi)\in\wh{H}^k(X; \R/A)$ is given by
$$s(\varphi)(z)=\int_{c}\omega+\int_{z'}\varphi\mod A,$$
where $z\in Z_{k-1}(X)$.

Take a lift $T_f\in C^{k-1}(X; \R)$ of $f$ such that $T_f(z)=f(z)\mod A$. Then there
exists a cocycle $u\in Z^k(X; A)$ depending on $T$ such that
$$\delta T_f=\omega-u.$$
Note that the cohomology class $[u]=\delta_2(f)$ is independent of the choice of $T_f$.
Denote by $T_\varphi\in C^{k-1}(X; \R)$ and $u_\varphi\in Z^k(X; A)$ the corresponding
lift and cocycle of $s(\varphi)$ respectively. Since $\delta T_\varphi=\omega-u_\varphi$,
it follows that
$$u-u_\varphi=\delta(T_\varphi-T_f).$$
Thus $\delta_2(f)=[u]=[u_\varphi]=\delta_2(s(\varphi))$. Since $(\delta_1, \delta_2)(f-s
(\varphi))=0$, it follows from the third exact sequence in \cite[Theorem 1.1]{CS85} that
\dis{f-s(\varphi)\in\frac{H^{k-1}(X; \R)}{r(H^{k-1}(X; A))}\cong\frac{\Omega^{k-1}_{d=0}
(X)}{\Omega^{k-1}_A(X)}}, where the later isomorphism is induced by the de Rham
isomorphism. Thus $f-s(\varphi)=i_2(\alpha)$ for some \dis{\alpha\in
\frac{\Omega^{k-1}_{d=0}(X)}{\Omega^{k-1}_A(X)}}. Thus
\begin{displaymath}
f=s(\varphi+\alpha).\qedhere
\end{displaymath}
\end{proof}

\subsection{A lift of the Gauss--Bonnet--Chern theorem}

In this subsection we state and prove the GBC theorem for oriented Euclidean vector
bundles taking values in differential characters.
\begin{thm}\label{thm 1}
Let $E\to X$ be a real oriented vector bundle of rank $2k$ with a Euclidean metric $h^E$
and a Euclidean connection $\nabla^E$. If $v\in\Gamma(X, E)$ is any section, then
\begin{equation}\label{eq 3.1}
\wh{\chi}(E, h^E, \nabla^E)-v^*\wh{U}(E, h^E, \nabla^E)=i_2\bigg(\int^1_0a(k)T(v\wedge
e^{-\Omega_{t, v}})dt\bigg)
\end{equation}
in $\wh{H}^{2k}(X; \R/\Z)$.
\end{thm}
In the following proof we adopt the notations in \S2.3.
\begin{proof}
First of all we give a formula of the differential Euler class following \cite[\S 3]{CS85}.
The fibration $SE\to X$ with fibers $\sss^{2k-1}$ gives the following exact sequence
\cdd{H_{2k-1}(\sss^{2k-1}) @>>> H_{2k-1}(SE) @>\wt{\pi}_*>> H_{2k-1}(X) @>>> 0}
Thus for $z\in Z_{2k-1}(X)$, there exist $y\in Z_{2k-1}(SE)$ and $w\in C_{2k}(X)$ such
that
$$z=\wt{\pi}_*(y)+\partial w.$$
Then
\begin{displaymath}
\begin{split}
\wh{\chi}(E, h^E, \nabla^E)(z)&=\wh{\chi}(E, h^E, \nabla^E)(\wt{\pi}_*y+\partial w)\\
&=(\wt{\pi}^*\wh{\chi}(E, h^E, \nabla^E))(y)+\int_w\chi(\nabla^E)\mod\Z.
\end{split}
\end{displaymath}
Consider the sphere bundle $\wt{\pi}:SE\to X$ and pull $E\to X$ back by $\wt{\pi}$; i.e.,
\cdd{\wt{\pi}^*E @>>> E \\ @VVV @VV\pi V \\ SE @>>\wt{\pi}> X}
There exists a ``canonical" form $Q\in\Omega^{2k-1}(SE)$, which is natural, such that
$$\chi(\wt{\pi}^*\nabla^E)=dQ\textrm{ and }\int_{\sss^{2k-1}}Q=1.$$
See \cite[(3.2)]{CS85}. Since the structure group of the vector bundle $\wt{\pi}^*E
\to SE$ is $\SO(2k)$, by applying \cite[Proposition 2.8]{CS85} and
\cite[Proposition 3.6]{CS74}, we have \dis{(\wt{\pi}^*\wh{\chi}(E, h^E, \nabla^E))(y)=
\int_yQ\mod\Z}. Thus, as in \cite[(3.3)]{CS85}, $\wh{\chi}(E, h^E, \nabla^E)(z)$ is
given by
$$\wh{\chi}(E, h^E, \nabla^E)(z)=\int_yQ+\int_w\chi(\nabla^E)\mod\Z.$$

We express the differential form $Q\in\Omega^{2k-1}(SE)$ in terms of Berezin integral.
As noted in \cite[p.51]{Z01} (see also \cite[(2.6)]{B05}), by restricting (\ref{eq 2.4})
to $SE$ and noting that \dis{\lim_{t\to+\infty}T(e^{-\Omega_t})=0}, the right-hand side
of (\ref{eq 2.4}) converges when $t\to\infty$, and therefore
\begin{equation}\label{eq 3.2}
\chi(\wt{\pi}^*\nabla^E)=d\int^\infty_0a(k)T(\bold{x}e^{-\Omega_t})dt.
\end{equation}
Obviously the form \dis{\int^\infty_0a(k)T(\bold{x}e^{-\Omega_t})dt} is natural, so it
differs from $Q$ by an exact form by \cite[Proposition 3.6]{CS74}. It follows from
\cite[Proposition 2.8]{CS85} that
\begin{equation}\label{eq 3.3}
\wh{\chi}(E, h^E, \nabla^E)(z)=\int_y\bigg(\int^\infty_0a(k)T(\bold{x}e^{-\Omega_t})
dt\bigg)+\int_w\chi(\nabla^E)\mod\Z.
\end{equation}

Second we give a formula of the differential Thom class. As remarked in
\cite[\S2.4]{HS05}, the differential Thom class can be associated to oriented real
vector bundles with Euclidean metrics and Euclidean connections using the Mathai--Quillen
formalism. The formula of the pullback of the differential Thom class $v^*\wh{U}(E, h^E,
\nabla^E)\in\wh{H}^{2k}(X; \R/\Z)$ by an arbitrary but fixed section $v\in\Gamma(X, E)$
is given by
\begin{displaymath}
\begin{split}
(v^*\wh{U}(E, h^E, \nabla^E))(z)&=(v^*\wh{U}(E, h^E, \nabla^E))(\wt{\pi}_*y+\partial w)\\
&=\wt{\pi}^*v^*\wh{U}(E, h^E, \nabla^E)(y)+\int_wv^*U(h^E, \nabla^E)\mod\Z.
\end{split}
\end{displaymath}
By (\ref{eq 2.6}) and (\ref{eq 3.2}), we have
\begin{displaymath}
\begin{split}
\wt{\pi}^*v^*U(h^E, \nabla^E)&=\wt{\pi}^*\chi(\nabla^E)-\wt{\pi}^*d\bigg(\int^1_0a(k)T
(v\wedge e^{-\Omega_{t, v}})dt\bigg)\\
&=d\bigg(\int^\infty_0a(k)T(\bold{x}e^{-\Omega_t})dt-\int^1_0a(k)\wt{\pi}^*T(v\wedge
e^{-\Omega_{t, v}})dt\bigg).
\end{split}
\end{displaymath}
Since the differential form inside the parentheses is obviously natural, it follows
from the same argument in deriving (\ref{eq 3.3}) that $(v^*\wh{U}(E, h^E, \nabla^E))(z)$
is given by
\begin{equation}\label{eq 3.4}
\begin{split}
(v^*\wh{U}(E, h^E, \nabla^E))(z)&=\int_y\bigg(\int^\infty_0a(k)T(\bold{x}e^{-\Omega_t})
dt-\int^1_0a(k)\wt{\pi}^*T(v\wedge e^{-\Omega_{t, v}})dt\bigg)\\
&~~~~+\int_wv^*U(h^E, \nabla^E)\mod\Z,
\end{split}
\end{equation}
Finally, we have
\begin{equation}\label{eq 3.5}
\begin{split}
&\qquad i_2\bigg(\int^1_0a(k)T(v\wedge e^{-\Omega_{t, v}})dt\bigg)(z)\\
&=\int_{\wt{\pi}_*y+\partial w}\bigg(\int^1_0a(k)T(v\wedge e^{-\Omega_{t, v}})dt\bigg)
\mod\Z\\
&=\int_y\bigg(\int^1_0a(k)\wt{\pi}^*T(v\wedge e^{-\Omega_{t, v}})dt\bigg)+\int_wd\bigg
(\int^1_0a(k)T(v\wedge e^{-\Omega_{t, v}})dt\bigg)\mod\Z.
\end{split}
\end{equation}
By comparing (\ref{eq 3.3}), (\ref{eq 3.4}), (\ref{eq 3.5}) and (\ref{eq 2.6}),
(\ref{eq 3.1}) holds.
\end{proof}
\bibliographystyle{amsplain}
\bibliography{MBib}
\end{document}